\documentclass[12pt]{article}
\usepackage{amsfonts,amsthm}
\usepackage{amssymb,amsmath}
\usepackage{enumerate}
\usepackage{amsmath}
\usepackage{mathtools}

\usepackage[dvipsnames]{xcolor}
\usepackage{hyperref}
\hypersetup{
    colorlinks=true,
    citecolor=blue,
    linkcolor=blue,
    filecolor=magenta,
    urlcolor=blue,
    pdftitle={Overleaf Example},
    pdfpagemode=FullScreen,
    }

\input{amssym.def}

\newtheorem{Definition}{Definition}[section]
\newtheorem{Theorem}[Definition]{Theorem}
\newtheorem{Lemma}[Definition]{Lemma}
\newtheorem{Proposition}[Definition]{Proposition}
\newtheorem{Corollary}[Definition]{Corollary}
\newtheorem{Example}[Definition]{Example}
\newtheorem{Remark}[Definition]{Remark}

\newcommand{\be}{\begin{equation}}
\newcommand{\ee}{\end{equation}}

\begin{document}

\title{\bf Graded $m$-nil clean ring}
\author{\bf Saikat Das \footnote {e-mail : saikatofficial607@gmail.com} \ and ~\bf Sukhendu Kar \footnote {e-mail : karsukhendu@yahoo.co.in} \\
{\small Department of Mathematics, Jadavpur University}\\
{\small 188, Raja S. C. Mallick Road, Kolkata -- 700032, India.}}
\date{}
\maketitle

\begin{abstract}
In this paper, we introduce the concept of graded $m$-nil clean ring to extend the existing notion of graded nil-clean ring introduced in $\cite{i}$. We explore fundamental properties of these rings, emphasizing the interplay between the identity component and the graded structure. We investigate certain conditions under which graded group rings, graded matrix rings and graded amalgamated algebras inherit the $m$-nil clean property from their components. Specifically, we establish sufficient conditions for graded group rings and graded matrix rings over commutative $m$-nil clean rings to be graded $m$-nil clean rings.  \end{abstract}

\noindent
\textbf{Keywords:} Graded rings, Graded amalgamted algebra, Group rings, Graded nil clean rings, Graded $m$-nil clean rings.

\noindent
\textbf {AMS Subject Classification :} $16U99$, $16Z05$, $16S34$.

\section{Introduction} \label{intro}

In 1977, the concept of clean ring was introduced by Nicholson \cite{n1}, where every element of the ring can be expressed as the sum of an idempotent and a unit. Since then, this idea has inspired several refinements and generalizations in multiple directions by many authors. Among them, the notions of \emph{nil-clean ring} \cite{d} and \emph{nil-good ring} \cite{d0} have attracted significant attention. In a nil-clean ring, each element decomposes as the sum of an idempotent and a nilpotent, while in a nil-good ring every element is either a nilpotent or the sum of a unit and a nilpotent. These classes of rings are closely related to the study of radicals, $\pi$-regularity of rings and K$\ddot{o}$the conjecture, a classical problem of ring theory. A further generalization is provided by the class of \emph{$m$-nil clean rings} \cite{mn}, in which every element can be written as the sum of an $m$-potent element and a nilpotent element. This framework naturally includes the class of nil-clean rings as the special case by considering $m=2$.

In addition to these nongraded investigations, graded analogues of nil-clean rings have been developed to better understand how the decomposition properties interact with group actions, matrix constructions and extensions. In $\cite{i}$, Ili\'c-Georgijevi\'c and \c{S}ahinkaya introduced the concepts of \emph{graded nil-clean rings} and \emph{graded strongly nil-clean rings}, while in \cite{c0}, Choulli, Mouanis and Namrok provided the graded version of nil-good rings. These works demonstrated that the graded setting gives rise to new structural phenomena that are not directly comparable to the nongraded case.

Motivated by these developments, we introduce and study the class of \emph{graded $m$-nil clean rings}.
A homogeneous element of a $G$-graded ring is called graded $m$-nil clean (resp. \emph{graded strongly $m$-nil clean}) if it can be expressed as the sum of a homogeneous $m$-potent and a homogeneous nilpotent (commuting in the strongly case). Thus our definition extends graded nil-clean rings (for the case $m=2$) and situates them within a broader family that parallels graded nil-good rings.

In this paper, we establish the fundamental properties of graded $m$-nil clean ring $R$, with particular emphasis on the connection between the identity component $R_e$ and the entire graded structure. We show that although $R_e$ being $m$-nil clean is necessary, it is not sufficient for $R$ to be graded $m$-nil clean and we provide sufficient conditions under which this implication holds. Furthermore, we investigate the lifting of $m$-potent elements modulo nil ideals and demonstrate that the graded $m$-nil clean property is preserved under quotients and certain extensions. We also introduce the class of \emph{graded strongly $m$-nil clean rings}, establish their relationship with strongly $\pi$-regular elements and provide characterizations in terms of strongly $\pi$-regular decompositions. Finally, we apply our results to amalgamated algebras, graded group rings and graded matrix rings. In particular, we establish certain necessary and sufficient conditions under which a group ring $RG$ inherits the graded $m$-nil clean property from $R$ and we also determine when graded matrix rings and graded amalgamated algebras are graded $m$-nil clean.

\section{Preliminaries and Prerequisites}
Let $R$ be a ring and $G$ a group with identity element $e$. Suppose there exists a family of additive subgroups $\{R_g\}_{g \in G}$ such that $R=\displaystyle \bigoplus_{g \in G} R_g$ and $R_gR_h \subseteq R_{gh}$ for all $g,h \in G$. Then $R$ is called a
\emph{$G$-graded ring}. The set $H=\displaystyle \bigcup_{g \in G} R_g$ is known as the \emph{homogeneous part} of $R$ and its members are called \emph{homogeneous elements}. Each subgroup $R_g$ is called a \emph{component}(or \emph{graded component}). If $a \in R_g$, we say that $a$ is homogeneous element of \emph{degree} $g$. The \emph{support of the graded ring} $R$ is defined as $\operatorname{supp}(R) = \{\, \sigma \in G \mid R_{\sigma} \neq 0 \,\}$. If $\operatorname{supp}(R)$ is finite, we write $\operatorname{supp}(R) < \infty$. In this case, $R$ is called a $G$-graded ring of finite support.

Given two $G$-graded rings $R$ and $S$, a ring homomorphism $f: R \longrightarrow S$ is said to be graded-homomorphism if $f(R_g) \subseteq f(S_g)$ for every $g \in G$. A right (resp. left, two-sided) ideal $I$ of a graded ring $R=\displaystyle \bigoplus_{g \in G} R_g$ is said to be \emph{homogeneous} (or \emph{graded}) whenever it decomposes as $I=\displaystyle \bigoplus_{g \in G} (I \cap R_g)$. When $I$ is a two-sided homogeneous ideal, the quotient ring $R/I$ inherits a natural $G$-grading given by $(R/I)_g=R_g/(I \cap R_g)$ as it's component. A homogeneous ideal $I$ is said to be \emph{graded-nil} provided that each homogeneous element of $I$ is nilpotent. A homogeneous right ideal $M$ of a graded ring $R$ is called a \emph{graded-maximal right ideal} if there is no proper homogeneous right ideal of $R$ properly containing $M$. A graded ring $R$ is termed \emph{graded-local} when it possesses a unique graded-maximal right ideal. The \emph{graded Jacobson radical} of a $G$-graded ring $R$, denoted by $J^g(R)$, is defined as the intersection of all graded-maximal right ideals of $R$. It is known that $J^g(R)$ forms a homogeneous two-sided ideal of $R$ (see \cite{n2}).

Let $R=\displaystyle \bigoplus_{g \in G} R_g$ be a $G$-graded ring. As noted in \cite{n1}, the associated group ring $RG$ admits a canonical $G$-grading. For each $g \in G$, the homogeneous component of degree $g$ is  $(RG)_g= \displaystyle \sum_{h \in G}~R_{gh^{-1}}h$ and with the multiplication defined by $(r_gg^{'})(r_hh^{'})=r_gr_h(h^{-1}g^{'}hh^{'})$, where $h,h^{'},g,g^{'} \in G$, $r_g \in R_g$ and $r_h \in R_h$.

Let $R$ be a $G$-graded ring and $n \in \mathbb{N}$. The matrix ring $M_n(R)$ admits a natural $G$-grading, defined with respect to a choice of $\overline{\sigma}=(g_1,\ldots,g_n) \in G^n$. For each $\lambda \in G$, consider the subset $$M_n(R)_{\lambda}(\overline{\sigma})= \{ (a_{ij}) \in M_n(R)~:~ a_{ij} \in {R_{g_i \lambda g^{-1}_j}}~ \text{for} ~ 1 \leq i,j  \leq n \}.$$

Then $M_n(R)=\displaystyle \bigoplus_{\lambda \in G} M_n(R)_{\lambda}(\overline{\sigma})$ is a $G$-graded ring with respect to usual matrix addition and multiplication. This ring is denoted by $M_n(R)_{\lambda}(\overline{\sigma})$.

In particular, if $\overline{\sigma}=(e,e,\ldots,e) \in G^n$, then  $M_n(R)_{\lambda}(\overline{\sigma})= \begin{pmatrix} R_\lambda & R_\lambda & \ldots & R_\lambda \\ R_\lambda & R_\lambda & \ldots & R_\lambda \\ \vdots & \vdots & \ldots & \vdots  \\R_\lambda & R_\lambda & \ldots & R_\lambda \end{pmatrix}$

\section{Graded $m$-nil clean ring}
This section is devoted to establish fundamental properties of graded $m$-nil clean rings which will be required in the subsequent development of this paper. We illustrate, by means of an example, that the class of graded nil clean rings is strictly contained within the class of graded $m$-nil clean rings. Furthermore, we show that the $m$-nil clean property of the identity component $R_e$ does not, in general, guarantee that the entire graded ring $R$ is $m$-nil clean.

\begin{Definition}
A homogeneous element $x$ of a $G$-graded ring $R$ is called \emph{graded $m$-nil clean} if $x=f+n$, where  $f$ is a homogeneous $m$-potent and $n$ is a homogeneous nilpotent in $R$. If, in addition, the condition $fn = nf$ holds, then $x$ is called a \emph{graded strongly $m$-nil clean} element. The ring $R$ itself is said to be graded $m$-nil clean (respectively, graded strongly $m$-nil clean) if every homogeneous element of $R$ is graded $m$-nil clean (respectively, graded strongly $m$-nil clean).\end{Definition}

\begin{Remark}
Let $\mathbb{F}$ be a finite field with $\mathbb{F} \neq \mathbb{Z}_2$ and $|\mathbb{F}| = m$. Suppose that $G = \{e, g\}$ be a cyclic group of order $2$.
Consider the ring 
$R =
\begin{pmatrix}
\mathbb{F} & \mathbb{F} \\
0 & \mathbb{F}
\end{pmatrix}$, equipped with the $G$-grading given by
\[
R_e =
\begin{pmatrix}
\mathbb{F} & 0 \\
0 & \mathbb{F}
\end{pmatrix}
\quad \text{and} \quad
R_g =
\begin{pmatrix}
0 & \mathbb{F} \\
0 & 0
\end{pmatrix}.
\]
Then $R$ is a graded $m$-nil clean ring but $R$ is not a graded nil clean ring.

Since $\mathbb{F}$ is a finite field with $|\mathbb{F}| = m$, every element of $\mathbb{F}$ is $m$-potent. Indeed, if $a = 0$, then clearly $a^{m} = 0 = a$. If $a \neq 0$, then $a \in \mathbb{F}^{\times}$, and the multiplicative group $\mathbb{F}^{\times}$ has order $m-1$. Hence $a^{m-1} = 1$, which implies $a^{m} = a$. Therefore every element of $\mathbb{F}$ satisfies $a^{m} = a$, and thus $\mathbb{F}$ is an $m$-nil clean ring. Since the homogeneous component $R_e$ being isomorphic to a direct product of two copies of $\mathbb{F}$ is also $m$-nil clean. Furthermore, every element of $R_g$ is nilpotent. Thus, by the definition of a graded $m$-nil clean ring, it follows that $R$ is a graded $m$-nil clean ring.

We now show that $R$ fails to be graded nil clean. Since $\mathbb{F} \neq \mathbb{Z}_2$, the field $\mathbb{F}$ is not nil clean. Indeed, in a field the only nilpotent element is $0$ and the only idempotents are $0$ and $1$; hence a field is nil clean if and only if it is $\mathbb{Z}_2$. Let $a \in \mathbb{F}$ with $a \neq 0,1$. Then $a^m = a$, so $a$ is $m$-potent, but it is not idempotent. Consequently, the homogeneous element
\[
\begin{pmatrix}
a & 0 \\
0 & a
\end{pmatrix}
\in R_e
\]
is not nil clean.
Thus $R_e$ itself is not nil clean and hence by \cite[Remark 3.2]{i}, the ring $R$ is not graded nil clean.
\end{Remark}

\begin{Remark} \label{2.2}
Let $R=\displaystyle \bigoplus_{g \in G} R_g$ be a $G$-graded ring and assume that $R$ be a graded $m$-nil clean ring. Let $x$ be a nonzero homogeneous element in $R_g$. By definition, $x$ can be expressed as the sum of a homogeneous $m$-potent $f$ and a nilpotent element $n$, i.e. $x=f+n$. We claim that both $f$ and $n$ necessarily belong to the same homogeneous component $R_g$ as $x$. To see this, suppose that $f \in R_h$ for some $h \neq g$. The uniqueness of decomposition then forces $f=0$ and hence $x=n$, showing that $n \in R_g$. On the other hand, if $h=g$, then it immediately follows that $n=x-f \in R_g$. Therefore, in either case, the components $f$ and $n$ are contained within the same graded component as $x$. \end{Remark}

\begin{Proposition}\label{prop3.4}

$(i)$ Let $R=\displaystyle \bigoplus_{g \in G} R_g$ be a $G$-graded (strongly) $m$-nil clean ring. Then every homomorphic image of $R$ is a graded (strongly) $m$-nil clean ring.

$(ii)$ Let $\{R_i\}_{i \in I}$ be a collection of $G$-graded rings for $I=\{1,2,\ldots,k \}$. Then $R=\displaystyle \prod_{i\in I} R_i$ is a graded (strongly) $m$-nil clean ring if and only if each $R_i$ is a graded (strongly) $m$-nil clean ring for $i \in I$.
\end{Proposition}

\begin{Definition}
A group $G$ is called $m$-torsion free if for any $g \in G$, $g^m=e$ then $g=e$, where $e$ denotes the identity element of $G$.\end{Definition}

\begin{Example}
$(i)$. Let $G$ be a finite group and let $m$ be a prime number. If $m \nmid |G|$, then $G$ is $m$-torsion free.

$(ii)$. The additive abelian group $\mathbb{Z}$ is $m$-torsion free for any positive integer $m \geq 2$.\end{Example}

\begin{Lemma}\label{2.4}
Let $R=\displaystyle \bigoplus_{g \in G} R_g$ be a $G$-graded ring and $f \in R$ be a non zero $m$-potent (i.e $f^m=f$) such that $f \in R_h$ for some $h \in G$. Then $h^{m-1}=e$, where $e$ is the identity element of $G$.\end{Lemma}
\begin{proof}
Since $R$ is a $G$-graded ring and $f \in R_h$, it follows that $f=f^m=f \ldots f \in R_h\ldots  R_h \subseteq R_{h^m}$. Therefore, $0 \neq f \in R_h \cap R_{h^m}$. This implies that $h^m=h$. Hence $h^{m-1}=e$.\end{proof}

\begin{Corollary}\label{2.41}
Let $R=\displaystyle \bigoplus_{g \in G} R_g$ be a $G$-graded ring and $G$ be a $(m-1)$-torsion free group. Then every nonzero homogeneous $m$-potent element of $R$ lies in the degree-zero component $R_e$.\end{Corollary}

\begin{proof}
Let $f \in R_g$ be a non zero $m$-potent element. Then by Lemma \ref{2.4}, it follows that $g^{m-1}=e$. This implies that $g=e$ as $G$ is $~(m-1)$-torsion free. So $f \in R_e$.\end{proof}

\begin{Example}
Let $m$ be a prime number with $m \geq 3$. Consider the ring $R=M_2(\mathbb{Z}_m) $, which is graded by the group $\mathbb{Z}_2$ with the following grading $R_0=\begin{pmatrix} \mathbb{Z}_m & 0 \\ 0 & \mathbb{Z}_m \end{pmatrix}$ and $R_1=\begin{pmatrix} 0 & \mathbb{Z}_m \\ \mathbb{Z}_m & 0 \end{pmatrix}$. Since $m$ is odd, $\mathbb{Z}_2$ is not $(m-1)$-torsion free. Consider the matrix $\begin{pmatrix} 0 & 2 \\ 2 & 0  \end{pmatrix} \in R_1$, which is a homogeneous element satisfying $ {\begin{pmatrix} 0 & 2 \\ 2 & 0  \end{pmatrix}}^m= \begin{pmatrix} 0 & 2^m \\ 2^m & 0  \end{pmatrix}= \begin{pmatrix} 0 & 2 \\ 2 & 0  \end{pmatrix}$,  is a nonzero $m$-potent. However, it does not belong to the degree-zero component $R_0$.
\end{Example}

\begin{Proposition} \label{2.5}
  Let $R=\displaystyle \bigoplus_{g \in G} R_g$ be a $G$-graded (strongly) $m$-nil clean ring. Then

  $(i)$ $R_e$ is an (strongly) $m$-nil clean ring.

  $(ii)$ If the group $G$ is $(m-1)$-torsion free, then every homogeneous element of $R_g$ with $g \neq e$ is nilpotent. \end{Proposition}
\begin{proof}
$(i)$ We know that $R_e$ is a subring of the graded ring $R$. Since $R$ is a graded $m$-nil clean ring, for any $x \in R_e$, $x=f+n$ for some homogeneous $m$-potent $f$ and homogeneous nilpotent $n$. Then by Remark \ref{2.2}, it follows that $f, n \in R_e$. Hence $R_e$ is an $m$-nil clean ring.

$(ii)$ Let $r \in R_g$ with $g \neq e$. Since $R$ is a graded $m$-nil clean ring, $r$ can be written as $r = f + n$, where $f$ is a homogeneous $m$-potent and $n$ is a homogeneous nilpotent. Then from Remark \ref{2.2}, it follows that $f, n \in R_g$. Thus by applying Corollary \ref{2.41}, we conclude that $f = 0$. Hence $r$ is nilpotent.  \end{proof}

\begin{Remark}\label{r1}
 If $R=\displaystyle \bigoplus_{g \in G} R_g$ is a $G$-graded $m$-nil clean ring, then $R_e$ is also $m$-nil clean follows from Proposition \ref{2.5}. This yields to the question of when the following implication holds true :
    $$ R_e ~\text{is a $m$-nil clean ring} \implies R=\bigoplus_{g \in G} R_g ~\text{is a graded $m$-nil clean ring.}$$

The following example shows that the above implication does not hold, in general. \end{Remark}
\begin{Example}
  Consider $R=A[X]$, where $A$ is a $m$-nil clean ring and $R$ is a $\mathbb{Z}$-graded with $R_i=AX^i$ if $i \geq 0$ and $R_i=0$ if $i < 0$. Then $R_0=A$ is $m$-nil clean and $\mathbb{Z}$ is $(m-1)$-torsion free group, but the homogeneous element $X \in R_1$ is not nilpotent. This contradicts Proposition \ref{2.5}. Therefore, $R$ is not a graded $m$-nil clean ring. Therefore $R_0$ is $m$-nil clean but $R$ is not graded $m$-nil clean. \end{Example}

Next, we give some sufficient conditions for the above implication to be true. For this purpose, we require the following results.

The following result follows from \cite[Theorem 8]{kd} :

\begin{Lemma}  \label{lp}
 Let $R$ be a ring with $~m-1 \in U(R)$, where $m\geq 2$ and $I$ be a nil ideal of $R$. Then $m$-potents lift modulo $I$ i.e for any $x \in R$ such that $x^m-x \in I$, there exists an $m$-potent $f=f^m \in R$ with $f-x \in I$. \end{Lemma}

\begin{Theorem} \label{q}
 Let $G$ be a $(m-1)$-torsion free group, where $m\geq 2$ and $R=\displaystyle \bigoplus_{g \in G} R_g$ a G-graded ring such that $(m-1) \in U(R)$ and $I$ be a graded-nil ideal of $R$. Then $R$ is a graded $m$-nil clean ring if and only if so is $R/I$. \end{Theorem}
\begin{proof}
First suppose that $R$ is a graded $m$-nil clean ring. Then $R/I$ is also a graded $m$-nil clean by Proposition \ref{prop3.4} as it is a graded homomorphic image of $R$.

Conversely, let $R/I$ be a graded $m$-nil clean ring. We need to show that $R_e$ is an $m$-nil clean and for every $n \in R_g$, $g \neq e$ is nilpotent. Since $R/I$ is a graded $m$-nil clean ring, by Proposition $\ref{2.5}$, we find that $(R/I)_e$ is an $m$-nil clean ring. Let us assume that $r \in R_e$. Since $\overline{R}=R/I$ is a graded $m$-nil clean ring, so $\overline{r} \in (R/I)_e=R_e/(I \cap R_e)$ can be written as $\overline{r}=\overline{f}+\overline{n}$ for some $m$-potent $\overline{f} \in (R/I)_e$ and a nilpotent $n \in (R/I)_e$. Since $1 \in R_e$ and $R_e$ is a subring of $R$, it follows that $m-1 = (m-1)\cdot 1 \in R_e$. Also  $(m-1)$ is a unit in the ring $R_e$ and $I \cap R_e$ is a nil ideal of $R_e$. Thus by Lemma \ref{lp}, $\overline{f}$ can be lifted to an $m$-potent, say $f \in R_e$. On the other hand, $\overline{r-f}$ is nilpotent in $(R/I)_e$. Therefore, $r-f$ is a nilpotent element in $R_e$ and so $R_e$ is an $m$-nil clean ring. Now assume that $r \in R_g$, where $g \neq e$. Since $R/I$ is graded $m$-nil clean, it follows from Proposition \ref{2.5} that $\overline{r} \in (R/I)_g$ is nilpotent. Since $I$ is a graded-nil ideal of $R$, $r$ is a nilpotent element of $R$, which completes the proof. \end{proof}

\begin{Proposition}
 Let $R=\displaystyle \bigoplus_{g \in G} R_g$ be a $G$-graded $m$-nil clean ring. Then $J^g(R)$ is a graded-nil ideal of $R$. \end{Proposition}
\begin{proof}
    Suppose that $R$ is a graded $m$-nil clean ring. Then by Proposition \ref{2.5}, $R_e$ is an $m$-nil clean ring. Thus by \cite[Proposition 1.2]{mn}, we find that $J(R_e)$ is a nil ideal of $R$. Also by \cite[Corollary 4.2]{c1}, $J(R_e)=J^g(R) \cap R_e$. Let $a$ be a homogeneous element of $J^g(R)$. Now if $a \in R_e$ then $a \in J(R_e)$ and so $a$ is nilpotent. On the other hand, if $a \in R_g$, where $g \neq e$, then Proposition \ref{2.5}, every homogeneous element of $R_g$ is nilpotent. Hence $a$ is nilpotent in this case as well. Therefore, $J^g(R)$ is a graded nil ideal of $R$.   \end{proof}

\begin{Corollary}
Let $G$ be a $(m-1)$-torsion free group, where $m\geq 2$ and $R=\displaystyle \bigoplus_{g \in G} R_g$ a G-graded ring such that $(m-1) \in U(R)$. Then $R$ is a graded $m$-nil clean ring if and only if $R/J^g(R)$ is a graded $m$-nil clean and $J^g(R)$ is a graded nil ideal of $R$. \end{Corollary}

Now we are ready to establish some sufficient conditions under which the implication in Remark~\ref{r1} holds. To this end, we first recall the definition of a $PI$-ring \cite{k0}.

A ring $R$ is a $PI$-ring if there is for some natural integer $n$, a nonzero element $P$ of $\mathbb{Z}[X_1,X_2, \ldots,X_n]$ such that $P(r_1,r_2,\ldots,r_n)=0$ for all $(r_1,r_2,\ldots,r_n) \in R^n$.

\begin{Theorem}
Let $R=\displaystyle \bigoplus_{g \in G} R_g$ be a $G$-graded ring. Suppose that one of the following conditions is satisfied :

$(i)$. The group  $G$ is finite and $R_gR_{g^{-1}}=0$ for every $g (\neq e) \in G$.

$(ii)$.  The ring $R$ is a $PI$-ring that is graded-local and $G$ is finite $(m-1)$-torsion free group such that $(m-1) \in U(R)$ and the order of $G$ is a unit in $R$.

Then $R$ is a graded $m$-nil clean ring  if $R_e$ is an $m$-nil clean ring.  \end{Theorem}

\begin{proof}
$(i)$. Suppose that $|G| = n$ and let $r \in R_h$, where $h \neq e$. Since $R$ is a $G$-graded ring, it follows that $r^k \in R_{h^k}$ for every positive integer $k$. In particular, $r^{n-1} \in R_{h^{\,n-1}}$. Because $G$ is a finite group of order $n$, we have $h^n = e$, and hence $h^{\,n-1} = h^{-1}$. Therefore, $r^{n-1} \in R_{h^{-1}}$. Consequently,
$r^n = r r^{n-1} \in R_h R_{h^{-1}}$. By the hypothesis that $R_g R_{g^{-1}} = 0$ for every $g \neq e$, it follows that $R_h R_{h^{-1}} = 0$. Hence $r^n = 0$, so $r$ is nilpotent. This shows that every homogeneous element of $R_h$ with $h \neq e$ is nilpotent and by assumption $R_e$ is an $m$-nil clean ring. Therefore $R$ is a graded $m$-nil clean.

$(ii)$. Let $R_e$ be an $m$-nil clean ring. Then from \cite[Proposition 1.2]{mn}, it follows that $J(R_e)$ is nil.
Hence the quotient $R_e / J(R_e)$ is also $m$-nil clean, as it is a homomorphic image of $R_e$.

As $|G| \in U(R)$, it follows by \cite[Theorem 4.4]{c1} that $J^{g}(R) = J(R)$. Moreover, from \cite[Theorem 3]{k}, it follows that the Jacobson radical $J(R)$ is a nil ideal and consequently $J^g(R)$ is a graded-nil ideal.

Let $H$ denote the homogeneous part of the quotient ring $R / J^g(R)$. Then by \cite[Theorem 3.27]{i}, we obtain an isomorphism $H \cong {R_e / J(R_e)}$.
Thus it follows that every homogeneous element of $R / J^g(R)$ is $m$-nil clean. Finally, by applying Theorem~\ref{q}, we conclude that $R$ is a graded $m$-nil clean ring.\end{proof}

A homogeneous element $a$ of a $G$-graded ring is said to be graded $\pi$-regular (see \cite{i}) if $a=f+u$ for some homogeneous idempotent $f$ and homogeneous unit $u$ with $af=fa$ and $faf$ is nilpotent. See \cite{i}, the uniqueness of a graded strongly $\pi$-regular decomposition holds in a $G$-graded ring.

\begin{Proposition}\label{mn}
Any strongly $m$-nil clean element in a ring $R$ is a strongly $\pi$-regular element in $R$.\end{Proposition}
\begin{proof}
    Let $a$ be strongly $m$-nil clean in $R$. Then $a=f+n$ with $fn=nf$, where $f$ is an $m$-potent and $n$ is a nilpotent in $R$. Thus $a=(1-f^{m-1})+(v+n)$, where $v=f+f^{m-1}-1$ is a unit. Since $fn = nf$, it follows that $f^{m-1}n = nf^{m-1}$, and hence $vn = nv$. Because $v$ is a unit commuting with the nilpotent element $n$, the element $v + n$ is also a unit in $R$.

Moreover, since $f^{m-1}$ commutes with $n$, we obtain
$a(1 - f^{m-1}) = (1 - f^{m-1})a$. Further, $(1 - f^{m-1})a(1 - f^{m-1})=(1 - f^{m-1})n(1 - f^{m-1})
= (1 - f^{m-1})n$. As $n$ is nilpotent and $1 - f^{m-1}$ commutes with $n$, the element $(1 - f^{m-1})n$ is nilpotent in $R$. Therefore, $a = (1 - f^{m-1}) + (v + n)$
is a strongly $\pi$-regular decomposition of $a$ in $R$.\end{proof}

Since every graded strongly $m$-nil clean element of a graded ring $R$ is, in particular, a strongly $m$-nil clean element of $R$, Proposition~\ref{mn} implies that every graded strongly $m$-nil clean element of $R$ is strongly $\pi$-regular in $R$. Thus we have the implication
\[
\begin{array}{c}
\text{graded strongly $m$-nil} \\
\text{ clean element}
\end{array}
\;\Longrightarrow\;
\begin{array}{c}
\text{strongly $m$-nil } \\
\text{clean element}
\end{array}
\;\xRightarrow{\text{Proposition~\ref{mn}}}\;
\begin{array}{cc}
\text{strongly $\pi$-regular } \\
\text{element.}
\end{array}
\]

\begin{Remark}
However, the graded analogue of the above implication does not necessarily hold. In general,
\[
\text{graded strongly $m$-nil clean}
\;\not\Longrightarrow\;
\text{graded strongly $\pi$-regular}.
\]

To see this, consider the ring $R=\begin{pmatrix}
    \mathbb{Z}_{2} & \mathbb{Z}_{2} \\ 0 & \mathbb{Z}_{2} \end{pmatrix}$ endowed with a $\mathbb{Z}_{2}$-grading as
\[
R_{0}=
\begin{pmatrix}
\mathbb{Z}_{2} & 0 \\
0 & \mathbb{Z}_{2}
\end{pmatrix}, 
\qquad
R_{1}=
\begin{pmatrix}
0 & \mathbb{Z}_{2} \\
0 & 0
\end{pmatrix}.
\]
Then the homogeneous element 
$\begin{pmatrix}
0 & 1 \\
0 & 0
\end{pmatrix}$ is graded strongly $2$-nil clean, but it is not graded $\pi$-regular.
\end{Remark}

\begin{Proposition}
Let $R=\displaystyle \bigoplus_{g \in G} R_g$ be a $G$-graded $m$-nil clean ring. Then the identity component $R_e$ is strongly $\pi$-regular.
\end{Proposition}

\begin{proof}
Since $R$ is a graded $m$-nil clean ring, it follows from Proposition~\ref{2.5}(i) that $R_e$ is an $m$-nil clean ring. Hence, by Proposition~\ref{mn}, every element of $R_e$ is strongly $\pi$-regular. Therefore, $R_e$ is a strongly $\pi$-regular ring.
\end{proof}

\begin{Proposition}\label{2.16}
Let $R$ be a ring and $a \in R$. Suppose that $a$ is strongly $\pi$-regular with strongly $\pi$-regular decomposition $a=f+u$. Then $a$ is strongly $m$-nil clean element in $R$ if and only if there exists an $m$-potent $g \in R$ commuting with $f$ and $u$ such that $f-g+u$ is nilpotent.\end{Proposition}
\begin{proof}
    Suppose that $a$ is a strongly $m$-nil clean element in $R$. Then $a=g+n$ for some $m$-potent $g$ and nilpotent $n$ with $gn=ng$. Then $a=(1-g^{m-1})+(g+g^{m-1}-1+n)$ is a strongly $\pi$-regular decomposition, by Proposition \ref{mn}. Then by the uniqueness of the strongly $\pi$-regular decomposition (see \cite[Proposition 2.6]{d}), it follows that $f=1-g^{m-1}$ and $u=g+g^{m-1}-1+n$. Hence, $f-g+u=(1-g^{m-1})-g+(g+g^{m-1}-1+n)=n$, which is nilpotent in $R$. Moreover, since $g$ commutes with $n$, it follows that $g$ commutes with both $f$ and $u$.
       
    Conversely, suppose that there exists an $m$-potent element $g \in R$ commuting with $f$ and $u$ such that $f-g+u$ is nilpotent in $R$. Since $a=f+u$, we can write $a=g+(f-g+u)$. Because $f-g+u$ is nilpotent and $g$ commutes with $f$ and $u$, it follows that $g$ commutes with $f-g+u$. Hence $a$ can be expressed as the sum of an $m$-potent element and a commuting nilpotent element. Therefore $a$ is strongly $m$-nil clean in $R$.
    \end{proof}

We now establish the graded analogue of the preceding results, which serves as a useful tool in the characterization of graded strongly $m$-nil clean elements in graded rings. 

\begin{Proposition}\label{2.17}
Let $R=\displaystyle \bigoplus_{g \in G} R_g$ be a $G$-graded ring, where $G$ is $(m-1)$-torsion free for $m\geq 2$. Suppose $a \in R$ be a graded strongly $\pi$-regular element with graded strongly $\pi$-regular decomposition $a=f+u$. Then $a$ is graded strongly $m$-nil clean in $R$ if and only if there exists a homogeneous $m$-potent element $g \in R_e$ such that $g$ commutes with both $f$ and $u$, $f-g+u$ is nilpotent in $R$, and $u \in R_e$. \end{Proposition}

\begin{proof}
Let $a$ be a graded strongly $m$-nil clean element in $R$. Suppose $0 \neq a \in R_h$, where $h \neq e$. Then by Corollary \ref{2.41}, it follows that $f=0$. Thus it follows that $a=u$ is a unit in $R$. Since $G$ is $(m-1)$-torsion free, so by Proposition \ref{2.5}, $a$ is nilpotent, which is impossible. Therefore, $a$ must belong to $R_e$. Then $a=g+n$ for some $m$-potent $g$ and nilpotent $n$ in $R_e$ with $gn=ng$. This yields the decomposition $a=(1-g^{m-1})+(g+g^{m-1}-1+n)$ which is a graded strongly $\pi$-regular decomposition of $a$ in $R_e$. By the uniqueness of such a decomposition (see \cite[Proposition 2.6]{d}), it follows that $f=1-g^{m-1}$ and $u=g+g^{m-1}-1+n$. Hence $u \in R_e$. Moreover for this choice of homogeneous $m$-potent $g$, the element $f - g + u$ is nilpotent in $R$.

Conversely, suppose that there exists a homogeneous $m$-potent element $g \in R_e$ which commutes with both $f$ and $u$, and such that $f-g+u$ is nilpotent in $R$, where $u \in R_e$. Since $G$ is $(m-1)$-torsion free, it follows from Corollary~\ref{2.41} that $f \in R_e$. Consequently, $a \in R_e$, and hence $f-g+u$ is a homogeneous nilpotent element of $R_e$. Therefore, $a=g+(f-g+u)$, where $g$ is a homogeneous $m$-potent element and $f-g+u$ is a homogeneous nilpotent element commuting with $g$. Hence $a$ is a graded strongly $m$-nil clean element of $R$.
\end{proof}

\section{Graded $m$-nil clean extensions}
Let $A = \displaystyle \bigoplus_{g \in G} A_g$ and $B = \displaystyle \bigoplus_{g \in G} B_g$ be two commutative $G$-graded rings.
Let $J = \displaystyle \bigoplus_{g \in G} J_g$ be a homogeneous ideal of $B$ and let  $f : A \longrightarrow B$ be a graded ring homomorphism.

Following \cite{g}, the amalgamation of $A$ with $B$ along $J$ with respect to $f$, denoted by $R = A \bowtie^{f} J$, is a $G$-graded ring with decomposition $R = \displaystyle \bigoplus_{g \in G} R_g$, where each homogeneous component is given by $R_g = (A \bowtie^{f} J)_g
= \{\, (a_g,\, f(a_g) + j_g) : a_g \in A_g,\; j_g \in J_g \,\}$.

\begin{Theorem}
Let $A$ and $B$ be two $G$-graded rings, where $G$ is a $(m-1)$-torsion free for $m\geq 2$ and $f : A \longrightarrow B$ be a graded ring homomorphism. Let $J$ be a homogeneous ideal of $B$. Then the following conditions are equivalent :

  $(i)$. $A \bowtie ^f J$ is a graded $m$-nil clean ring.

  $(ii)$. $A$ and $f(A)+J$ are graded $m$-nil clean rings.
\end{Theorem}
\begin{proof}
 Let $A \bowtie ^f J$ be a graded $m$-nil clean ring.  Then by \cite[Theorem 3.5]{g}, it follows that $A$ and $f(A)+J$ are homomorphic image of $A \bowtie ^f J$. Therefore, $A$ and $f(A)+J$ are graded $m$-nil clean rings.

 Conversely, suppose that $(a,f(a)+j) \in (A \bowtie ^f J)_g$ with $g \neq e$. Then $a \in A_g$ and $j \in J_g$. Since the group $G$ is $(m-1)$-torsion free and both $A$, $f(A)+J$ are graded $m$-nil clean rings, by Proposition \ref{2.5} it follows that $a$ and $j$ are nilpotent. Consequently, $(a,f(a)+j)$ is nilpotent. Now consider an element $(a,f(a)+j) \in (A \bowtie ^f J)_e$, where $a \in A_e$ and $j \in J_e$. Since $f$ is degree preserving ring homomorphism, it follows that $f(a)+j \in (f(A)+J)_e$. By \cite[Theorem 1.10.]{mn}, we have  both $a^m-a$ and $(f(a)+j)^m-(f(a)+j)$ are nilpotent. Consequently, $(a,f(a)+j)^m-(a,f(a)+j) \in (A \bowtie ^f J)_e$ is nilpotent. Then by \cite[Theorem 1.10]{mn}, it follows that $(A \bowtie ^f J)_e$ is an $m$-nil clean ring. Hence $A \bowtie^f B$ is a graded $m$-nil clean ring.
\end{proof}

Next we deal with the graded $m$-nil clean property of graded group rings.

Recall that a group $G$ is said to be a locally finite group if every finitely generated subgroup of $G$ is finite. A group $G$ is called a torsion group if every element of $G$ has finite order. If $p$ is prime, $g \in G$ is called the $p$-torsion element if the order of $g$ is $p^k$ for some $k \geq 1$. The group $G$ is called $p$-group if every element of $G$ is $p$-torsion.

\begin{Lemma} \label{3.1}
Let $R$ be a $G$-graded ring such that $(m-1)\in U(R)$ for $m\geq 2$. Suppose that $G$ is $(m-1)$-torsion free, locally finite $p$-group, where $p$ is nilpotent in $R$. If $R$ is a graded $m$-nil clean ring, then the group ring $RG$ is also a graded $m$-nil clean ring. \end{Lemma}
\begin{proof}
Let $x \in RG$. We aim to show that $x$ is graded $m$-nil clean. Since $G$ is locally finite, there exists a finite subgroup $H$ of $G$ such that $x \in RH$. 
Therefore it suffices to consider the case when $G$ is a finite group. Since $p$ is nilpotent in $R$, from \cite[Theorem 9]{c2}, it follows that the augmentation ideal $\Delta(RG)$ is nilpotent. Furthermore, the quotient $RG/\Delta(RG)$ and $R$ are isomorphic as rings. Since $(m-1) \in U(R)=U((RG)_e)$, by Theorem \ref{q}, we find that $RG$ is a graded $m$-nil clean ring.\end{proof}

\begin{Theorem}
    Let $R=\displaystyle \bigoplus_{g \in G} R_g$ be a $G$-graded ring.

    $(i)$. Let $R$ be a graded $m$-nil clean ring with $p \in Nil(R)$ and $G$ be a locally finite $p$-group, where $p$ is prime and $p$ divides $m$. Then $RG$ is a graded $m$-nil clean ring.

    $(ii)$. Let $R$ be a graded ring such that $m$-potents and nilpotents of $R$ are all homogeneous. If $RG$ is a graded $m$-nil clean ring, then $R$ is a graded $m$-nil clean ring. \end{Theorem}
\begin{proof}
$(i)$. Since $p$ is nilpotent and $p$ divides $m$, the element $m$ is nilpotent in $R$.
Therefore, $m-1$ is a unit in $R$. Assume, for the sake of contradiction, that there exists an element $g \in G$ such that $g^{\,m-1} = e$.
In this case, $p^{k}$ divides $m-1$ for some positive integer $k$, which would force $m-1$ to be nilpotent.
However, a nilpotent element cannot be invertible, contradicting the fact that $m-1$ is a unit in $R$. Hence the only possible solution is $g = e$ and thus $G$ is $(m-1)$-torsion free. The rest part then follows directly from Lemma~\ref{3.1}.

$(ii)$ Since $RG$ is a graded $m$-nil clean ring, it follows from Proposition~\ref{2.5} that $(RG)_e$ is an $m$-nil clean ring. By \cite[Proposition 2.1 (4)]{n}, the map $f : R \longrightarrow (RG)_e$ defined as $f(\displaystyle \sum_{g \in G}~r_g)=\displaystyle\sum_{g \in G}~r_gg^{-1}$, is a ring isomorphism. Consequently, $R$ is a $m$-nil clean ring and hence graded $m$-nil clean ring. \end{proof}


\begin{Theorem}
Let $G$ be a $(m-1)$-torsion free group for $m\geq 2$ and let $R=\displaystyle \bigoplus_{g \in G} R_g$ be a $G$-graded commutative ring of finite support such that $J(R) \subseteq J^g(R)$ and $(m-1) \in U(R)$. If $R$ is a graded $m$-nil clean ring, then $M_n(R)(\overline{\sigma})$ is also a graded $m$-nil clean ring, where $\overline{\sigma}=(e,e,\ldots,e) \in G^n$.\end{Theorem}

\begin{proof}
Since $R$ is commutative, every homogeneous nilpotent element necessarily lies in $J(R)$. Moreover by \cite[Lemma 12]{b}, the graded Jacobson radical $J^g(R)$ contains all homogeneous elements of $J(R)$. Thus it follows that $J^g(R)$ contains every homogeneous nilpotent element of $R$. Therefore, we have every homogeneous element of  $\overline{R}=R/J^g(R)$ is $m$-potent. Now consider any non zero homogeneous element $A \in M_n(\overline{R})(\overline{\sigma})$. Since the group $G$ is $(m-1)$ torsion free, we obtain that $A$ belongs to the subring $M_n(\overline{R}_e)$. Moreover as $\overline{R}_e$ is an $m$-nil clean ring, \cite[Theorem 3.4]{mn} ensures that $M_n(\overline{R}_e)$ is also an $m$-nil clean ring. Consequently, $M_n(\overline{R})(\overline{\sigma})$ is a graded $m$-nil clean ring.

Consider the natural homomorphism $f : M_n(R) \longrightarrow M_n(\overline{R})$ defined by $f([a_{ij}])=[\overline{a}_{ij}]$. This yields that $M_n(\overline{R})(\overline{\sigma}) \cong M_n(R)/M_n(J^g(R))$ and so $M_n(R)/M_n(J^g(R))$ is a graded $m$-nil clean ring. Moreover for any natural number $n$, we have $J^g(M_n(R))=M_n(J^g(R))$. Since $R_e$ is an $m$-nil clean ring, by \cite[Theorem 3.4]{mn}, it follows that
$M_n(R)_e(\overline{\sigma})=M_n(R_e)$ is an $m$-nil clean ring. This implies that $J(M_n(R)_e(\overline{\sigma}))$ is nil, by \cite[ Proposition 1.2.]{mn}. Also by the Amitsur–Levitski theorem (see \cite{a}), $M_n(R)$ is a $PI$-ring. Thus by \cite[Theorem 3]{k}, $J(M_n(R))$ is nil. Since $R$ has finite support, it follows from \cite[Corollary 2.9.4]{n2} that $J^g(M_n(R)) \subseteq J(M_n(R))$. Hence $J^g(M_n(R))$ is a graded nil ring. Therefore, by Theorem \ref{q} $M_n(R)(\overline{\sigma})$ is a graded $m$-nil clean ring.\end{proof}

Let $A$ be a ring and let $R=M_n(A)$.
For $1 \leq i,j \leq n$, we denote by $E_{ij}$ the matrix having $1$ in the $(i,j)$-th position and $0$ elsewhere.
It is well known that $R$ admits a natural $\mathbb{Z}$-grading defined as follows : for each $t \in \mathbb{Z}$,
\[
R_t =
\begin{cases}
0, & |t| \geq n, \\[6pt]
\displaystyle \sum_{i=1}^{\,n-t} A E_{i,i+t}  , & 0 \leq t < n, \\[10pt]
\displaystyle \sum_{i=-t+1}^{\,n} A E_{i,i+t}  , & -n < t < 0.
\end{cases}
\]
Thus $R$ decomposes into a $\mathbb{Z}$-graded structure with homogeneous components determined by the diagonals of the matrix.

\begin{Theorem}
    A ring $A$ is an $m$-nil clean ring if and only if the matrix ring $M_n(A)$, equipped with its natural $\mathbb{Z}$-grading, is a $\mathbb{Z}$-graded $m$-nil clean ring. \end{Theorem}

\begin{proof} Let $R=M_n(A)$ be $\mathbb{Z}$-graded $m$-nil clean ring. For any element $a \in A$, the scalar matrix $aI_n \in R_0$, which is an $m$-nil clean ring. Therefore, $a$ is an $m$-nil clean element in $A$. Hence $A$ is an $m$-nil clean ring.

Conversely, consider an arbitrary homogeneous element of $M_n(A)$. If it lies in a graded component $R_t$ with $t \neq 0$, then the matrix is strictly upper or lower triangular with zero diagonal entries and hence nilpotent. If it lies in $R_0$, then it is a diagonal matrix whose diagonal entries come from $A$. Since $A$ is an $m$-nil clean ring, each such diagonal entry is $m$-nil clean and hence the entire matrix is homogeneous $m$-nil clean. Thus the result follows.
\end{proof}

\begin{Theorem}
  Let $G$ be a $(m-1)$-torsion free group for $m\geq 2$ and let $R=\displaystyle \bigoplus_{g \in G} R_g$ be a $G$-graded ring such that $(m-1) \in U(R)$. Then $R$ is a graded $m$-nil clean ring if and only if the triangular matrix ring $T_n(R)(\overline{\sigma})$ is a graded $m$-nil clean ring for every $\overline{\sigma} \in G^n$.
\end{Theorem}
\begin{proof}
Consider the ideal $I \subseteq T_n(R)(\overline{\sigma})$ consisting of all matrices whose principal diagonal entries are zero. Then by \cite[Theorem 3.20.]{i} it is a homogeneous nilpotent ideal and according to Theorem \ref{q}, $T_n(R)(\overline{\sigma})$ is a graded $m$-nil clean ring if and only if the quotient ring $T_n(R)(\overline{\sigma})/I$ is a graded $m$-nil clean ring. But the quotient ring $T_n(R)(\overline{\sigma})/I$ is isomorphic to $R \times R \times \cdots \times R$ ($n$ copies). Then by Proposition \ref{prop3.4} $(i)$,  $T_n(R)(\overline{\sigma})/I$ is a graded $m$-nil clean ring. Thus it follows that $T_n(R)(\overline{\sigma})$ is a graded $m$-nil clean ring if and only if $R$ is a graded $m$-nil clean ring. \end{proof}

\noindent
\textbf{\large Acknowledgment}

\noindent
The first author acknowledges the financial support received as a Senior Research Fellow (SRF) from the University Grants Commission (UGC), Government of India (Award Letter No. 211610081351).



\end{document}